\documentclass[a4paper, 11pt]{article}
\newcommand{\thetitle}{Honest adaptive confidence bands and self-similar 
functions}

\newcommand{\forenames}{Adam D.}
\newcommand{\surname}{Bull}
\newcommand{\fullname}{\forenames\ \surname}

\newcommand{\mscone}{62G15}
\newcommand{\msctwo}{62G07}
\newcommand{\mscthree}{62G08}
\newcommand{\mscfour}{62G20}
\newcommand{\themsclass}{\mscone\ (Primary); \msctwo, \mscthree, \mscfour\ 
(Secondary)}

\newcommand{\kwdone}{nonparametric statistics}
\newcommand{\kwdtwo}{adaptation}
\newcommand{\kwdthree}{confidence sets}
\newcommand{\kwdfour}{supremum norm}
\newcommand{\kwdfive}{self-similar functions}
\newcommand{\thekeywords}{\kwdone, \kwdtwo, \kwdthree, \kwdfour, \kwdfive}

\newcommand{\addressone}{Statistical Laboratory}
\newcommand{\addresstwo}{University of Cambridge}

\newcommand{\theemail}{a.bull@statslab.cam.ac.uk}

\newcommand{\theabstract}{
Confidence bands are confidence sets for an unknown function \(f,\) containing 
all functions within some sup-norm distance of an estimator. In the density 
estimation, regression, and white noise models, we consider the problem of 
constructing adaptive confidence bands, whose width contracts at an optimal 
rate over a range of H\"{o}lder classes.  

While adaptive estimators exist, in general adaptive confidence bands do not, 
and to proceed we must place further conditions on \(f.\) We discuss previous 
approaches to this issue, and show it is necessary to restrict \(f\) to 
fundamentally smaller classes of functions.

We then consider the self-similar functions, whose H\"{o}lder norm is similar 
at large and small scales. We show that such functions may be considered 
typical functions of a given H\"{o}lder class, and that the assumption of 
self-similarity is both necessary and sufficient for the construction of 
adaptive bands.  Finally, we show that this assumption allows us to resolve 
the problem of undersmoothing, creating bands which are honest simultaneously 
for functions of any H\"{o}lder norm.
}

\usepackage[round]{natbib}
\usepackage{amsmath, amsfonts, amssymb, amsthm, mathtools, thmtools, booktabs, 
tikz, mparhack}
\usepackage[bookmarks, breaklinks, colorlinks, linkcolor=blue, citecolor=blue, 
pdfauthor={\fullname}, pdftitle={\thetitle}]{hyperref}

\newcommand{\comment}[1]{}
\let\oldmarginpar\marginpar
\renewcommand\marginpar[1]{\-\oldmarginpar[\raggedleft\footnotesize 
#1]{\raggedright\footnotesize #1}}

\DeclarePairedDelimiter{\abs}{\lvert}{\rvert}
\DeclarePairedDelimiter{\norm}{\lVert}{\rVert}
\newcommand{\N}{\mathbb{N}}
\newcommand{\Z}{\mathbb{Z}}

\newcommand{\R}{\mathbb{R}}
\renewcommand{\P}{\mathbb{P}}
\newcommand{\E}{\mathbb{E}}
\newcommand{\Var}{\mathbb{V}\mathrm{ar}}

\newcommand{\iid}{\overset{\mathrm{i.i.d.}}{\sim}}

\numberwithin{equation}{section}
\declaretheorem[numberwithin=section]{theorem}
\declaretheorem[sibling=theorem]{lemma}

\declaretheorem[sibling=theorem]{assumption}
\declaretheorem[sibling=theorem]{proposition}

\newcommand{\jmin}{j_n^{\min}}
\newcommand{\jmax}{j_n^{\max}}
\newcommand{\ja}{j_n^{\text ad}}
\newcommand{\jc}{j_n^{\text cl}}
\newcommand{\je}{j_n^{\text ex}}
\newcommand{\lu}{\overline \lambda}
\renewcommand{\ll}{\underline \lambda}
\newcommand{\jl}{\underline{j}_n^{\text ad}}
\newcommand{\ju}{\overline{j}_n^{\text ad}}
\newcommand{\Jc}{J_n^{\text cl}}
\newcommand{\Je}{J_n^{\text ex}}
\newcommand{\jh}{\hat j_n}
\newcommand{\jha}{\jh^{\text ad}}

\newcommand{\jhc}{\jh^{\text cl}}
\newcommand{\jhe}{\jh^{\text ex}}
\newcommand{\nln}{c_{n, 1}}
\newcommand{\bjk}[1][]{\beta_{j_{#1},k_{#1}}}
\newcommand{\bjkh}[1][]{\hat \beta_{j_{#1},k_{#1}}}
\newcommand{\blkh}[1][]{\hat \beta_{l_{#1},k_{#1}}}
\newcommand{\Mh}{\hat M_n}
\newcommand{\sh}{\hat s_n}
\newcommand{\sph}{(s + 1/2)}
\newcommand{\shph}{(\sh + 1/2)}
\newcommand{\smin}{s_{\min}}
\newcommand{\smax}{s_{\max}}

\newcommand{\fh}{\hat f}
\newcommand{\fb}{\bar f}
\newcommand{\sfb}{\overline \sigma_\varphi}
\newcommand{\Ra}{R_n^{\text ad}}
\newcommand{\Ca}{C_n^{\text ad}}
\renewcommand{\Re}{R_n^{\text ex}}
\newcommand{\Ce}{C_n^{\text ex}}

\begin{document}

\title{\thetitle
\footnotetext{\emph{Mathematics subject classification 2010.} \themsclass}
}
\author{\fullname\\\footnotesize \addressone \\\footnotesize \addresstwo \\
\footnotesize \theemail}
\date{}
\maketitle

\begin{abstract}
  \theabstract
\end{abstract}

\section{Introduction}
\label{sec:introduction}

Suppose we have an unknown function \(f : [0, 1] \to \R\) we wish to estimate.  
Our data may come from:
\begin{enumerate}
  \item density estimation, where \(f\) is a density on \([0, 1],\) and we 
    observe
    \[X_1, \dots, X_n \iid f;\]
  \item fixed design regression, where we observe
    \[Y_i \coloneqq f(x_i) + \varepsilon_i, \qquad \varepsilon_i \iid N(0, 
    \sigma^2),\]
    for \(x_i \coloneqq i/n,\) \(i=1, \dots, n;\) or
  \item white noise, where we observe the process
    \[Y_t \coloneqq \int_0^t f(s)\, ds + n^{-1/2} B_t,\]
    for a standard Brownian motion \(B.\)
\end{enumerate}

The performance of an estimator \(\fh_n\) depends on the smoothness of the 
function \(f.\) In the following, we will measure performance by the 
\(L^\infty\) loss, \(\norm{\fh_n - f}_\infty,\) where \(\norm{f}_\infty 
\coloneqq \sup_{x \in [0, 1]} \abs{f(x)}.\) \(L^\infty\) loss is the hardest 
of the \(L^p\) loss functions to estimate under, but provides intuitive risk 
bounds, simultaneously describing local and global performance.  If the 
function \(f\) is known to lie in the smoothness class \(C^s(M)\) of functions 
with \(s\)-H\"{o}lder norm at most \(M,\)
\begin{multline*}
  C^s(M) \coloneqq \bigg\{f \in C([0, 1]) : f \text{ has } k \coloneqq \lceil 
  s \rceil - 1 \text{ derivatives, }\\
   \norm{f}_\infty, \dots, \norm{f^{(k)}}_\infty \le M, \sup_{x, y \in [0, 1]} 
   \frac{\abs{f^{(k)}(x)-f^{(k)}(y)}}{\abs{x-y}^{s-k}} \le M\bigg\},
\end{multline*}
then the \(L^\infty\) minimax rate of estimation,
\[\inf_{\fh_n} \sup_{f \in C^s(M)} \E_f \norm{\fh_n-f}_\infty,\]
decays like \((n/\log n)^{-s/(2s+1)}\) 
\citep[see][]{tsybakov_introduction_2009}.

The simplest estimators attaining this rate depend on the quantities \(s\) and 
\(M,\) which in practise we will not know in advance. However, it is possible 
to estimate \(f\) {\em adaptively}: to choose an estimator \(\fh_n,\) not 
depending on \(s\) or \(M,\) which nevertheless obtains the minimax rate over 
a range of classes \(C^s(M),\)
\[\sup_{f \in C^s(M)} \E_f \norm{\fh_n - f}_\infty = O\left((n/\log 
n)^{-s/(2s+1)}\right).\]
Techniques for constructing such estimators include Lepskii's method 
\citep{lepskii_problem_1990}\comment{lepskii_optimal_1997}, wavelet 
thresholding \citep{donoho_wavelet_1995}, and model selection 
\citep{barron_risk_1999}.

Of course, to make full use of an adaptive estimator \(\fh_n,\) we must also 
quantify the uncertainty in our estimate. We would like to have a risk bound 
\(R_n,\) depending only on the data, which satisfies \(\norm{f-\fh_n}_\infty 
\le R_n\) with high probability. Equivalently, we would like a {\em confidence 
band},
\begin{equation}
  \label{eq:band-defn}
  C_n \coloneqq \{f \in C([0, 1]) : \norm{f-\fh_n}_\infty \le R_n\},
\end{equation}
containing \(f\) with high probability. To benefit from the adaptive nature of 
\(\fh_n,\) we would also like the radius \(R_n\) to be adaptive, decaying at a 
rate \((n/\log n)^{-s/(2s+1)}\) over any class \(C^s(M).\)

Unfortunately, this is impossible in general 
\citep{low_nonparametric_1997,cai_adaptation_2004}.
The size of an adaptive confidence band must depend on the parameters \(s\) 
and \(M,\) which we cannot estimate from the data: the function \(f\) may be 
{\em deceptive}, superficially appearing to belong to one smoothness class 
\(C^s(M),\) while instead belonging to a different, rougher class. If we wish 
to proceed, we must place further conditions on \(f.\)

Different conditions have been considered by \citet{picard_adaptive_2000}, 
\citet{genovese_adaptive_2008}, \citet{gine_confidence_2010}, and 
\citet{hoffmann_adaptive_2011}. Of note, \citeauthor{gine_confidence_2010} 
place a self-similarity condition on \(f,\) requiring its regularity to be 
similar at large and small scales; they then obtain confidence bands which 
contract adaptively over classes \(C^s(M),\) where \(M > 0\) is fixed. 
\citeauthor{hoffmann_adaptive_2011} consider a weaker separation condition, 
which allows adaptation to finitely many classes \(C^{s_1}(M), \dots, 
C^{s_k}(M).\)

The conditions in these two papers are qualitatively different. In 
\citet{hoffmann_adaptive_2011}, the family of functions \(f\) under 
consideration at time \(n\) asymptotically contains the full model,
\begin{equation}
\label{eq:finite-model}
\mathcal F \coloneqq \bigcup_{i=1}^k C^{s_i}(M),\qquad 0 < s_1 < \dots < 
s_k,\,M>0.
\end{equation}
The confidence bands constructed are thus eventually valid for all functions 
\(f \in \mathcal F,\) although the time \(n\) after which a band is valid 
depends on the unknown \(f.\)
The penalty for this generality comes in the nature of the adaptive result: 
the bands contract at rates \(n^{-s_i/(2s_i+1)}\) for any \(f \in 
C^{s_i}(M),\) but they do not attain the minimax rate \(n^{-s/(2s+1)}\) for 
\(f \in C^s(M),\) \(s \not \in \{ s_1, \dots, s_k\}.\)

Conversely, in \citet{gine_confidence_2010}, the bands attain the rate 
\(n^{-s/(2s+1)}\) for any \(f \in C^s(M),\) \(s \in [\smin, \smax].\) However, 
the family of functions considered does not, even in the limit, contain the 
full model,
\begin{equation}
\label{eq:continuum-model}
\mathcal F \coloneqq \bigcup_{s=\smin}^{\smax} C^s(M), \qquad 0 < \smin < 
\smax,\,M>0.
\end{equation}
Instead, some functions \(f\) must be permanently excluded from consideration.

We can describe this difference in terms of dishonest confidence sets.  We say 
a confidence set \(C_n\) for \(f\) is {\em honest}, at level \(1 - \gamma,\) 
if it satisfies
\begin{equation}
\label{eq:honest}
\limsup_n \sup_{f \in \mathcal F} \P_f(f \not\in C_n) \le \gamma,
\end{equation}
where \(\mathcal F\) is the entire family of functions \(f\) we wish to adapt 
to \citep[see][\comment{\S1, }and references therein]{robins_adaptive_2006}. 
Honesty is necessary to produce practical confidence sets; it ensures that 
there is a known time \(n,\) not depending on \(f,\) after which the level of 
the confidence set is not much smaller than \(1 - \gamma.\)  In contrast, a 
{\em dishonest} set satisfies the weaker condition
\[\sup_{f \in \mathcal F} \limsup_n \P_f(f \not \in C_n) \le \gamma.\]
While dishonest confidence sets are not useful for inference, they can provide 
a useful benchmark of nonparametric procedures. The bands in 
\citet{hoffmann_adaptive_2011} are dishonest confidence sets for the full 
model \eqref{eq:finite-model}; those in \citet{gine_confidence_2010} are not, 
for the model \eqref{eq:continuum-model}.

In the following, we will show that this distinction is intrinsic: that the 
problem of adapting to finitely many \(s_i\) is fundamentally different from 
adapting to continuous \(s.\) We will construct confidence bands which are 
adaptive in the model \eqref{eq:continuum-model}, under a weaker 
self-similarity condition than in \citet{gine_confidence_2010}; functions 
satisfying this condition may be considered typical members of any class 
\(C^s(M).\) We will then show that our condition is as weak as possible for 
adaptation over \eqref{eq:continuum-model}, and that no adaptive confidence 
band can be valid, even dishonestly, for all of \eqref{eq:continuum-model}.

We also provide further improvements on past results. Firstly, past 
constructions of adaptive confidence sets under self-similarity have required 
{\em sample splitting}: splitting the data into two groups, one for estimating 
the function \(f,\) and the other for estimating its smoothness. In the 
construction of our bands, we will show that this procedure can be avoided, 
leading to smaller constants in the rate of contraction.

More importantly, in past results \(M\) is assumed known; in general, this 
assumption is required to obtain meaningful results. However, in practise, we 
will not know \(M\) in advance; we would much prefer to adapt also to the 
unknown H\"{o}lder norm.  We would thus like a confidence band which is valid 
even for the model
\begin{equation*}
\mathcal F \coloneqq \bigcup_{M=0}^\infty\bigcup_{s = \smin}^{\smax} C^s(M), 
\qquad 0 < \smin < \smax.
\end{equation*}

In \citet{gine_confidence_2010}, the authors suggest the standard remedy of 
undersmoothing: constructing bands valid for subsets of \(C^s(M_n),\) with 
\(M_n \to \infty\) as \(n \to \infty.\) However, doing so not only incurs a 
rate penalty; it also gives a dishonest band. We will instead show that, under 
the assumption of self-similarity necessary for adaptation, we can perform 
honest inference without an a priori bound on \(M.\)

We would therefore like to construct a confidence band for \(f \in C^s(M),\) 
which:
\begin{enumerate}
\item is adaptive;
\item makes assumptions on \(f\) as weak as possible; and
\item is honest simultaneously for a range of \(s,\) and all \(M > 0.\)
\end{enumerate}
Confidence sets \(C_n\) in the literature are often constructed to be {\em 
asymptotically exact}, satisfying
\[\sup_{f \in \mathcal F} \, \abs{\P_f(f \not \in C_n) - \gamma} \to 0\]
as \(n \to \infty.\) We will show that, using an undersmoothed estimator, we 
can construct an exact confidence band, satisfying conditions (ii) and (iii), 
which is rate-adaptive up to a logarithmic factor.

We will argue, however, that in this case exactness may be undesirable. 
Instead, we will construct an {\em inexact} confidence band, satisfying only 
\eqref{eq:honest};
while we no longer know the exact level of our confidence band, this level is 
guaranteed to be at least \(1 - \gamma\). Our inexact band is centred at an 
adaptive Lepskii-type estimator, is asymptotically smaller, more likely to 
contain the function \(f,\) and satisfies all three conditions (i)--(iii).

As our bands cannot rely on a known (or unknown) bound on the H\"{o}lder norm 
\(M,\) their construction differs significantly from those given previously in 
the literature. We likewise describe new approaches to undersmoothing, and to 
linking the white noise model with density estimation and regression. In each 
case, rather than assuming \(M\) is bounded, we must make fundamental use of 
the self-similarity property of our functions \(f.\)

Our bands thus depend on self-similarity parameters \(\varepsilon\) and 
\(\rho,\) which determine the functions \(f\) to be excluded. In this sense, 
they are no different than any other technique, whether fixing a class 
\(C^s(M)\) in advance, or using one of the methods discussed previously. (The 
bands in \citealp{gine_confidence_2010}, do not require a choice of parameters 
to construct, but they are honest only over families \(\mathcal F\) which do; 
using them in practise would thus involve an implicit choice of parameters.) 
The advantage in our bands is that, while we must still exclude some functions 
\(f,\) we do so only where necessary for adaptation.

The parameters \(\varepsilon\) and \(\rho\) may in practise be set by 
domain-specific knowledge, or by convention, as is common with the confidence 
level \(1 - \gamma = 95\%.\) Whether this is suitable for practical inference 
is a matter for further study. We leave the reader, however, with the words of 
Box: ``all models are wrong, but some are useful.'' 

In \autoref{sec:self-similarity}, we describe our self-similarity condition, 
and in \autoref{sec:adaptation}, we state our main results. We provide proofs 
in \hyperref[app:self-similarity]{Appendices 
\ref{app:self-similarity}}--\ref{app:negative}.

\section{Self-similar functions}
\label{sec:self-similarity}

To state our results, we must first define our self-similarity condition. We 
will need a wavelet basis of \(L^2([0,1]);\) for an introduction to wavelets, 
and their role in  statistical applications, see 
\citealp{haerdle_wavelets_1998}. We begin with  \(\varphi\) and \(\psi\), the 
scaling function and wavelet of an orthonormal multiresolution analysis on 
\(L^2(\R).\) We make the following assumptions on  \(\varphi\) and \(\psi,\) 
which are satisfied, for example, by Daubechies wavelets and symlets, with \(N 
\ge 6\) vanishing moments (\citealp[\S6.1]{daubechies_ten_1992}; 
\citealp[\S14]{rioul_simple_1992}).

\begin{assumption}\ \label{ass:wavelet-basis}
\begin{enumerate}
\item For \(K \in \N,\) \(\varphi\) and \(\psi\) are supported on the interval 
  \([1-K, K].\)
\item For \(N \in \N,\) \(\psi\) has \(N\) vanishing moments:
\[\int_\R x^i \psi(x) \, dx = 0, \qquad i=0,\dots,N-1.\]
\item \(\varphi\) is twice continuously differentiable.
\end{enumerate}
\end{assumption}

Using the construction of \citet{cohen_wavelets_1993}, we can then generate an 
orthonormal wavelet basis of \(L^2([0,1]),\) with basis functions 
\[\varphi_{j_0,k}, \quad k = 0, \dots, 2^{j_0}-1,\]
and
\[\psi_{j,k}, \quad j > j_0, \, k = 0, \dots, 2^j-1,\]
for some suitable lower resolution level \(j_0 > 0.\) (See also 
\citealp{chyzak_construction_2001}.) For \(k \in [N, 2^j - N),\) the basis 
functions are given by scalings of \(\varphi\) and \(\psi,\)
\[\varphi_{j_0, k}(x) \coloneqq 2^{j_0/2}\varphi(2^{j_0}x - k), \qquad 
\psi_{j, k} \coloneqq 2^{j/2}\psi(2^jx-k).\]
For other values of \(k,\) the basis functions are specially constructed, so 
as to form an orthonormal basis of \(L^2([0,1]),\) with desired smoothness 
properties.

Using this wavelet basis, we may proceed to define the spaces \(C^s\) over 
which we wish to adapt. Given a function \(f \in L^2([0, 1]),\)
\[f = \sum_k \alpha_k \varphi_{j_0,k} + \sum_{j>j_0}\sum_k \beta_{j,k} 
\psi_{j,k},\]
for \(s \in (0, N),\) define the \(C^s\) norm of \(f\) by
\[\norm{f}_{C^s} \coloneqq \max\left(\sup_k \,\abs{\alpha_k}, \, \sup_{j,\,k} 
\, 2^{j(s+1/2)}\abs{\beta_{j,k}}\right).\]
Define the spaces
\[C^s \coloneqq \{f \in L^2([0, 1]) : \norm{f}_{C^s} < \infty\},\]
and for \(M > 0,\)
\[C^s(M) \coloneqq \{f \in L^2([0, 1]): \norm{f}_{C^s} \le M\}.\]

For \(s \not \in \N,\) these spaces are equivalent to the classical H\"{o}lder 
spaces; for \(s \in \N,\) they are equivalent to the Zygmund spaces, which 
continuously extend the H\"{o}lder spaces \citep[\S4]{cohen_wavelets_1993}.
In either case, we may therefore take this to be our definition
of \(C^s\) in the following.

We are now ready to state our self-similarity condition. Denote the wavelet 
series of \(f,\) for resolution levels \(i\) to \(j,\) \(i > j_0,\) by
\[f_{i,j} \coloneqq \sum_{l=i}^j \sum_k \beta_{l,k} \psi_{l,k},\]
and for \(i = j_0,\) by
\[f_{j_0,j} \coloneqq \sum_k \alpha_k \varphi_{j_0,k} + f_{j_0 + 1, j}.\]
Fix some \(\smax \in (0, N);\) for \(s \in (0, \smax),\) \(M > 0,\) 
\(\varepsilon \in (0, 1),\) and \(\rho \in \N,\) we will say a function \(f 
\in C^s(M)\) is {\em self-similar}, if
\begin{equation}
  \label{eq:self-similar}
  \norm{f_{j, \rho j}}_{C^s} \ge \varepsilon M \ \forall \ j \ge j_0.
\end{equation}
If \(s = \smax,\) we will instead require \eqref{eq:self-similar} only for \(j 
= j_0.\) Denote the set of self-similar \(f \in C^s(M)\) by \(C^s_0(M, 
\varepsilon, \rho );\) for fixed \(\varepsilon,\) \(\rho ,\) we will denote 
this set simply as \(C^s_0(M).\)

The above condition ensures that the regularity of \(f\) is similar at small 
and large scales, and will be shown to be necessary to perform adaptive 
inference. To bound the bias of an adaptive estimator \(\fh_n,\) we need to 
know the regularity of \(f\) at small scales, which we cannot observe. If 
\(f\) is self-similar, however, we can infer this regularity from the 
behaviour of \(f\) at large scales, which we can observe.

Similar conditions have been considered by previous authors, in the context of 
turbulence by \citet{frisch_singularity_1985} and 
\citet{jaffard_frisch-parisi_2000},
and more recently in statistical applications by \citet{picard_adaptive_2000} 
and \citet{gine_confidence_2010}. We can show that condition 
\eqref{eq:self-similar} is weaker than the condition in 
\citeauthor{gine_confidence_2010}; we will see in \autoref{sec:adaptation} 
that it is, in a sense, as weak as possible.

\begin{proposition}
  \label{prop:weak-condition}
  Given \(\smin \in (0, \smax],\) \(b > 0,\) \(0 < b_1 \le b_2,\) and \(j_1 
  \ge j_0,\) there exist \(M > 0,\) \(\varepsilon \in (0, 1),\) and \(\rho  
  \in \N\) such that, for any \(s \in [\smin, \smax],\) the condition
  \begin{equation}
    \label{eq:gn-condition}
    f \in C^s \cap C^{\smin}(b), \qquad b_12^{-js} \le \norm{f_{j+1, 
    \infty}}_\infty \le b_22^{-js} \ \forall \ j \ge j_1,
  \end{equation}
  implies \(f \in C^s_0(M, \varepsilon, \rho).\)  Conversely, given \(s \in 
  (0, \smax],\) \(M > 0,\) \(\varepsilon \in (0, 1),\) and \(\rho  > 1,\) 
  there exist \(f \in C^s_0(M, \varepsilon, \rho)\) which do not satisfy the 
  above condition, for any \(\smin \in (0, s],\) \(b > 0,\) \(0 < b_1 \le 
  b_2,\) and \(j_1 \ge j_0.\)
\end{proposition}

In fact, we can show that self-similarity is a generic property:  that the set 
\(\mathcal D\) of self-dissimilar functions, which for some \(s\) never 
satisfy \eqref{eq:self-similar}, is in more than one sense negligible.  
Firstly, \comment{as in \citet[\S3.5]{gine_confidence_2010},} we can show that 
\(\mathcal D\) is nowhere dense: the self-dissimilar functions cannot 
approximate any open set in \(C^s(M).\) In particular, this means that 
\(\mathcal D\) is meagre.  Secondly, \comment{as in 
\citet[\S2.5]{hoffmann_adaptive_2011},} we can show that \(\mathcal D\) is a 
null set, for a natural probability measure \(\pi\) on \(C^s(M).\) We thus 
have that \(\pi\)-almost-every function in \(C^s(M)\) is self-similar.

\begin{proposition}
  \label{prop:negligible-set}
  For \(s \in (0, \smax]\) and \(M > 0,\)\comment{and \(\rho  \in \N,\)} 
  define
  \[\mathcal D \coloneqq C^s(M) \setminus \bigcup_{\varepsilon \in (0, 
  1),\,\rho \in \N} C^s_0(M, \varepsilon, \rho ).\]
  Further define a probability measure \(\pi\) on \(f \in C^s(M),\) with \(f\) 
  having independently distributed wavelet coefficients,
  \[\alpha_{k} \sim M2^{-j_0\sph}U([-1, 1]), \qquad \beta_{j,k} \sim 
  M2^{-j\sph}U([-1,1]).\]
  Then:
  \begin{enumerate}
    \item \(\mathcal D\) is nowhere dense in the norm topology of \(C^s(M);\) 
      and
    \item \(\pi(\mathcal D) = 0.\)
  \end{enumerate}
\end{proposition}

These results are given for the self-similarity condition 
\eqref{eq:gn-condition} in \citet[\S3.5]{gine_confidence_2010}, and 
\citet[\S2.5]{hoffmann_adaptive_2011}; as a consequence of 
\autoref{prop:weak-condition}, they hold for our condition 
\eqref{eq:self-similar} also. \comment{Similar arguments can apply to other 
measures placed on wavelet coefficients; in particular, nonparametric Bayesian 
priors of the above form place zero mass on the set \(\mathcal D.\)} We 
conclude that the self-similar functions may be considered typical members of 
any class \(C^s(M).\)

\section{Self-similarity and adaptation}
\label{sec:adaptation}

We are now ready to state our main results. First, however, we will require an 
additional assumption on our wavelet basis, allowing us to precisely control 
the variance of our estimators. This assumption is verified for 
Battle-Lemari\'{e} wavelets in \cite{gine_periodized_2011}; for compactly 
supported wavelets, the assumption is difficult to verify analytically, but 
can be tested with provably good numerical approximations. In 
\citet[\S3]{bull_smirnov-bickel-rosenblatt_2011}, the assumption is shown to 
hold for Daubechies wavelets and symlets, with \(N = 6, \dots, 20\) vanishing 
moments. Larger values of \(N,\) and other wavelet bases, can be easily 
checked, and the assumption is conjectured to hold also in those cases.

\begin{assumption}
  \label{ass:sigma-maximum} The 1-periodic function
  \[\sigma^2_\varphi(t) \coloneqq \sum_{k \in \Z} \varphi(t - k)^2\]
  attains its maximum \(\overline \sigma^2_\varphi\) at a unique point \(t_0 
  \in [0, 1),\) and \((\sigma^2_\varphi)''(t_0) < 0.\)
\end{assumption}

We may now construct a confidence band which, under self-similarity, is exact, 
honest for all \(M > 0,\) and contracts at a near-optimal rate.  We centre the 
band at an undersmoothed estimate of \(f\):
an estimate slightly rougher than optimal, chosen so that the known variance 
dominates the unknown bias (as in \citealp{hall_effect_1992}, for example). 
This allows us to construct an asymptotically exact confidence band, although 
the larger variance leads to a logarithmic rate penalty. We state our results 
for the white noise model, which serves as an idealisation of density 
estimation and regression; we will return later to consequences for the other 
models.

\begin{theorem}
  \label{thm:exact-band}

  In the white noise model, fix \(0 < \gamma < 1,\) \(\smin \in (0, \smax],\) 
  and set
  \[r_n(s) \coloneqq (n/\log n)^{-s/(2s+1)}\log n, \qquad \mathcal F \coloneqq 
  \bigcup_{s \in [\smin, \smax], \, M > 0}C^s_0(M).\]
  There exists a confidence band \(\Ce \coloneqq \Ce(\gamma, \smin, \smax, 
  \varepsilon, \rho)\) as in \eqref{eq:band-defn}, with radius \(\Re,\)
  satisfying:
  \begin{enumerate}
    \item \(\sup_{f \in \mathcal F} \abs{\P(f \not\in \Ce) - \gamma} \to 0;\) 
      and
    \item for a fixed constant \(L > 0,\) and any \(s \in [\smin, \smax],\) 
      \(M > 0,\)
      \[\sup_{f \in C^s_0(M)} \P_f\left(\Re > LM^{1/(2s+1)} r_n(s) \right) \to 
      0.\]
  \end{enumerate}
\end{theorem}

We can do better by dropping the requirement of exactness. Intuitively, we may 
feel that an exact band should be preferable: given an inexact band, surely we 
can modify it to produce something more accurate?
In fact, this is not necessarily the case. Consider a simplified statistical 
model, where we wish to identify a parameter \(\theta \in \R,\) and have the 
luxury of observing data \(X = \theta.\)
The optimal confidence set for \(\theta\) is thus \(\{X\},\) but this set is 
not exact at the 95\% level. We can produce an exact set by adding noise: if 
\(Z \sim N(0, 1),\) the confidence set
\[\{x \in \R : \abs{X+Z - x} \le \Phi^{-1}(0.975)\}\]
is exact at the 95\% level. It is also clearly inferior. The perfect, inexact 
set is preferable to the imperfect, exact one.

The situation is similar in nonparametrics. We can undersmooth, adding noise 
to produce an exact band, but in doing so we make our band both asymptotically 
larger, and less likely to contain the function \(f.\) In practise, this is 
clearly undesirable. Instead, we will give one of the main results of this 
paper: we will provide an inexact band, centred at an adaptive Lepskii-type 
estimator, which under self-similarity is honest over a larger family of 
functions, and exact rate-adaptive with respect to \(s\) and \(M.\)

\begin{theorem}
  \label{thm:adaptive-band}
  In the white noise model, fix \(0 < \gamma < 1,\) and set
  \[r_n(s) \coloneqq (n/\log n)^{-s/(2s+1)}, \qquad \mathcal F \coloneqq 
  \bigcup_{s \in (0, \smax], M > 0} C^s_0(M).\]
  There exists a confidence band \(\Ca \coloneqq \Ca(\gamma, \smax, 
  \varepsilon, \rho)\) as in \eqref{eq:band-defn}, with radius \(\Ra,\) 
  satisfying:
  \begin{enumerate}
    \item \(\limsup_n \sup_{f \in \mathcal F} \P(f \not\in \Ca) \le \gamma;\) 
      and
    \item for a fixed constant \(L > 0,\) and any \(s \in (0, \smax],\) \(M > 
      0,\)
      \[\sup_{f \in C^s_0(M)} \P_f\left(\Ra > 
      \frac{LM^{1/(2s+1)}}{2^s-1}r_n(s) \right) \to 0.\]
  \end{enumerate}
\end{theorem}

The constant in the above rate contains an extra \(1/(2^s-1)\) term, which is 
present to allow for \(s\) tending to 0. Note that if, as before, we restrict 
to \(s \ge \smin > 0,\) we may then fold this term into the constant \(L,\)  
producing a rate of the same form as in \autoref{thm:exact-band}.

As is standard, the rates adapt only to smoothnesses \(s \le \smax;\) if \(f\) 
is smoother than our wavelet basis, we cannot reliably detect this from the 
wavelet coefficients. However, our self-similarity condition 
\eqref{eq:self-similar} is weaker when \(s = \smax,\) and the class 
\(C^{\smax}_0(M)\) contains many smoother functions \(f;\)
in this case we obtain the rate of contraction optimal for \(C^{\smax}(M).\)

\autoref{thm:adaptive-band} is, in more than one sense, maximal.  Firstly, we 
can verify that the minimax rate of estimation over \(C^s_0(M)\) is the same 
as over \(C^s(M).\) Since any adaptive confidence band must be centred at an 
adaptive estimator, we may conclude that the above results are indeed optimal.

\begin{theorem}
  \label{thm:minimax-rate}
  In the white noise model, fix \(0 < \gamma < \frac12,\) \(s \in (0, 
  \smax],\) \(M > 0.\) An estimator \(\fh_n\) cannot satisfy
  \[\limsup_n \sup_{f \in C^s_0(M)} \P_f\left(\norm{\fh_n -f}_\infty \ge 
  r_n\right) \le \gamma,\]
  for any rate \(r_n = o\left( (n/\log n)^{-s/(2s+1)} \right).\)
\end{theorem}

Secondly, we can show that the self-similarity condition 
\eqref{eq:self-similar} is, in a sense, as weak as possible. In 
\eqref{eq:self-similar}, the function \(f\) is required to have significant 
wavelet coefficients on resolution levels \(j\) growing at most geometrically. 
If we relax this assumption even slightly, allowing the significant 
coefficients to occur less often, then adaptive inference is impossible.

For \(s \in (0, \smax),\) \(M > 0,\) denote by \(C^s_1(M)\) the set of \(f \in 
C^s(M)\) satisfying the slightly weaker self-similarity condition,
\[\norm{f_{j,\rho _jj}}_{C^s} \ge \varepsilon M \ \forall \ j \ge j_0,\]
for fixed \(\varepsilon > 0,\) and \(\rho _j \in \N,\) \(\rho _j \to \infty.\)
Even allowing dishonesty, and with known bound \(M\) on the H\"{o}lder norm, 
we cannot construct a confidence band which adapts to classes \(C^s_1(M).\)

\begin{theorem}
  \label{thm:minimal-conditions}
  In the white noise model, fix \(0 < \gamma < \frac12,\) \(0 < s_{\min} < 
  s_{\max},\) and \(M > 0.\) Set
  \[r_n(s) \coloneqq (n/\log n)^{-s/(2s+1)}, \qquad \mathcal F \coloneqq 
  \bigcup_{s \in (\smin, \smax)} C^s_1(M).\]
  A confidence band \(C_n,\) with radius \(R_n,\) cannot satisfy:
  \begin{enumerate}
    \item \(\limsup_n \P_f(f \not\in C_n) \le \gamma,\) for all \(f \in 
      \mathcal F;\) and
    \item \(R_n = O_p(r_n(s))\) under \(\P_f,\) for all \(f \in C^s_1(M),\) 
      \(s \in (\smin, \smax).\)
  \end{enumerate}
\end{theorem}

As a consequence, we firstly cannot adapt to the full classes \(C^s(M).\) More 
importantly, we cannot, as in \citet{hoffmann_adaptive_2011}, obtain 
adaptation merely by removing elements of the classes \(C^s(M)\) which are 
asymptotically negligible. In order to construct adaptive bands, we must fully 
exclude some functions \(f\) from consideration, and this remains true even 
when \(M\) is known.

The difference between these problems lies in the accuracy to which we must 
estimate \(s.\) To distinguish between finitely many classes, we need to know 
\(s\) only up to a constant; to adapt to a continuum of smoothness, we must 
know it with error shrinking like \(1/\log n.\) The finite-class problem is in 
this sense more like the \(L^2\) adaptation problem studied in 
\citet{bull_adaptive_2011}; the distinctive nature of the \(L^\infty\) 
adaptation problem is revealed only when requiring adaptation to continuous 
\(s.\)

While the above theorems are stated for the white noise model, we can prove 
similar results for density estimation and regression. The following theorem 
gives a construction of adaptive bands in these models; other results can be 
proved, for example, as in \citet{gine_confidence_2010}, and 
\citet{bull_adaptive_2011}.

\begin{theorem}
  \label{thm:other-models}
  In the density estimation model, let \(\smin \in (0, \smax],\) or in the 
  regression model, \(\smin \in [\tfrac12, \smax].\) In either model, the 
  statement of \autoref{thm:adaptive-band} remains true, for the family
  \[\mathcal F \coloneqq \bigcup_{s \in [\smin, \smax], \, M > 0}C^s_0(M),\]
  and with constants \(L,\) \(L'\) depending on \(s\) and \(M.\)
\end{theorem}

\subsection*{Acknowledgements}

We would like to thank Richard Nickl for his valuable comments and 
suggestions.

\appendix

\section{Results on self-similarity}
\label{app:self-similarity}

We begin by establishing that our self-similarity condition 
\eqref{eq:self-similar} is weaker than \eqref{eq:gn-condition}, the condition 
in \citet{gine_confidence_2010}.


\begin{proof}[Proof of \autoref{prop:weak-condition}]
  We first consider the case \(s < \smax.\) Given \eqref{eq:gn-condition}, for 
  \(j > j_1,\) \(k \in [N, 2^j - N),\) we obtain
  \[
    \abs{\beta_{j,k}} = \abs{\langle f_{j, \infty}, \psi_{j,k} \rangle}
    \le \norm{f_{j, \infty}}_\infty \norm{\psi_{j,k}}_1
    \le b_2 \norm{\psi}_1 2^{-j(s+1/2)},\]
  and similar bounds for \(k \in [0, N) \cup [2^j - N, 2^j).\) We thus 
  conclude \(f \in C^s(M),\) for a constant \(M > 0.\)

  We will choose \(\varepsilon \in (0, 1)\) small, \(\rho  \in \N\) large, so 
  that \(\rho j_0 \ge j_1,\) and
  \[C \coloneqq M(\varepsilon + 2^{-(\rho j_0 - j_1)s})\]
  is small.  If \(f \not \in C^s_0(M),\) we have \(j_2 \ge j_0\) such that
  \[\abs{\beta_{j,k}} < \varepsilon M 2^{-j(s+1/2)},\]
  for all \(j \in [j_2, \rho  j_2],\) \(k \in [0, 2^j).\)
  Let \(j_3 \coloneqq \max(j_1, j_2).\) Then
  \begin{align*}
    \norm{f_{j_3 + 1, \infty}}_\infty &\lesssim M\left(\sum_{j=j_3+1}^{\rho  
    j_2} \varepsilon 2^{-js} + \sum_{j=\rho  j_2 + 1}^\infty 2^{-js}\right)\\
    &\lesssim M\left(\varepsilon 2^{-j_3 s} + 2^{-\rho  j_2 s}\right)
    \lesssim C2^{-j_3 s},
  \end{align*}
  contradicting \eqref{eq:gn-condition} for \(C\) small. Thus, given 
  \eqref{eq:gn-condition}, we have \(M,\) \(\varepsilon,\) and \(\rho \) for 
  which \(f \in C^s_0(M).\)

  Conversely, given \(s \in (0, \smax],\) \(M > 0,\) \(\varepsilon \in (0, 
  1),\) and \(\rho  > 1,\) for \(i\in\N\) set \(j_i \coloneqq \rho ^ij_0,\) 
  and consider the function
  \[f \coloneqq \sum_{i=1}^\infty M2^{-j_i(s+1/2)} \psi_{j_i, 2^{j_i-1}}\]
  in \(C^s_0(M).\) We have
  \begin{align*}
    \norm{f_{j_n+1,\infty}}_\infty &\lesssim M\sum_{i=n+1}^\infty2^{-j_is}
    \lesssim 2^{-j_{n+1}s}
    =o(2^{-j_ns})
  \end{align*}
  as \(n \to \infty,\) so \(f\) does not satisfy \eqref{eq:gn-condition} for 
  any \(\smin,\) \(b,\) \(b_1,\) \(b_2,\) and \(j_1.\) As our self-similarity 
  condition is weaker for \(s = \smax,\) the same is true also in that case.
\end{proof}

\section{Constructing adaptive bands}
\label{app:construction}

To construct confidence bands satisfying the conditions in 
\autoref{sec:adaptation},
we will use estimators \(\fh_n\) given by truncated empirical wavelet 
expansions,
\[\fh(j_n) \coloneqq \sum_k \hat \alpha_k \varphi_{j_0,k} + \sum_{j_0 < j \le 
j_n}\sum_k \hat \beta_{j,k} \psi_{j,k},\]
for the empirical wavelet coefficients
\[\hat \alpha_k \coloneqq \int \varphi_{j_0,k}(t) \,dY_t, \qquad
\hat \beta_{j,k} \coloneqq \int \psi_{j,k}(t) \,dY_t.\]
We will centre our bands on adaptive estimators \(\fh(\jh),\) where the 
resolution level \(\jh\) also depends on \(Y\).

We will consider several different choices of resolution level, corresponding 
to different properties of the function \(f,\) and the class \(C^s(M)\) to 
which it belongs. We first consider the adaptive resolution choice \(\ja,\) 
chosen in terms of the function \(f.\) Pick sequences \(\jmin, \jmax \in \N,\) 
\(j_0 \le \jmin \le \jmax,\) so that \(2^{\jmin} \sim (n/\log n)^{1/(2N+1)},\) 
and \(2^{\jmax} \sim n/\log n.\) Further define
\begin{equation*}
  c_{n,\mu} \coloneqq (n/(\log n)^\mu)^{-1/2},
\end{equation*}
and for \(\kappa > 0,\) \(\mu \ge 1,\) let
\[\ja(\kappa, \mu) \coloneqq \sup \left(\{\jmin\} \cup \{\jmin < j \le \jmax : 
\sup \nolimits_k \,\abs{\bjk} \ge \kappa c_{n,\mu}\} \right).\]
While \(\ja\) is unknown, we can estimate it by a Lepskii-type resolution 
choice,
\begin{equation*}
  \jha(\kappa, \mu) \coloneqq \sup \left(\{\jmin\} \cup \{\jmin < j \le \jmax 
  : \sup \nolimits_k \,\abs{\bjkh} \ge \kappa c_{n,\mu}\} \right),
\end{equation*}
which depends only on the data.  Fix \(\lambda > \sqrt{2},\) \(\nu \ge 1,\) 
and for convenience set \(\jha \coloneqq \jha(\lambda, \nu).\) If \(\nu = 1,\) 
we will see \(\fh(\jha)\) is then an adaptive estimator of \(f;\) if \(\nu > 
1,\) it is near-adaptive.

While the above statements are true for general \(f,\) they do not provide us 
with an estimate of the error in \(\fh_n.\) To produce confidence bands, we 
must estimate the smoothness of \(f,\) and this is where self-similarity is 
required. We will consider values of the truncated H\"{o}lder norm,
\[M^s_{i,j} \coloneqq \norm{f_{i,j}}_{C^s},\]
which measures the smoothness of \(f\) at resolution levels \(i\) to \(j,\)
In a slight abuse of notation, set \(\beta_{j_0, k} \coloneqq \alpha_k,\) and 
\(\hat \beta_{j_0, k} \coloneqq \hat \alpha_k.\) (Note that \(\beta_{j_0, k}\) 
and \(\hat \beta_{j_0, k}\) are otherwise undefined, as the wavelets 
\(\psi_{j, k}\) exist only for \(j > j_0.\)) We may then bound \(M^s_{i,j}\) 
by the quantities
\begin{align*}
  \underline M^s_{i,j} &\coloneqq \sup_{i \le l \le j, k} 2^{l\sph} 
  (\abs{\blkh} - \sqrt{2}c_{n,1})^+,\\
  \overline M^s_{i,j} &\coloneqq \sup_{i \le l \le j, k} 2^{l\sph} 
  (\abs{\blkh} + \sqrt{2}c_{n,1}),
\end{align*}
and we will show in \autoref{app:constructive} that for \(j \le \jmax,\) 
\(M_{i,j}^s \in [\underline M^s_{i,j}, \overline M^s_{i,j}]\) with high 
probability.

Set \(j_1 = \rho j_0,\) \(j_2 = \lfloor \jha / \rho  \rfloor,\) \(j_3 = 
\jha,\) and suppose \(n\) is large enough that \(\jmin \ge \rho  j_1,\) so 
\(j_0 \le j_1 \le j_2 \le j_3.\) If \(f \in C^s_0(M)\) for \(s < \smax,\) then 
with high probability,
\[R(s) \coloneqq \frac{\overline M^s_{j_2,j_3}}{\underline M^s_{j_0,j_1}} \ge 
\frac{M^s_{j_2,j_3}}{M^s_{j_0,j_1}} \ge \varepsilon.\]
Assuming further \(s \ge \smin,\) for some \(\smin \ge 0,\) we can lower bound 
\(s\) by
\[\sh \coloneqq \inf( \{\smax\} \cup \{s \in [\smin, \smax) : R(s) \ge 
\varepsilon\}).\]
Since
\[R(s) = \frac {\overline M^s_{j_2,j_3}2^{-j_1\sph}} {\underline 
M^s_{j_0,j_1}2^{-j_1\sph}}\]
is increasing in \(s,\) \(\sh\) can be found efficiently using binary search.

Likewise, set
\[M(s) \coloneqq \varepsilon^{-1}\overline M^s_{j_0,j_1},\]
and \(\Mh \coloneqq M(\sh).\) With high probability,
\[M(s) 2^{-j_1\sph} \ge \varepsilon^{-1} M^s_{j_0,j_1}2^{-j_1\sph} \ge 
M2^{-j_1\sph},\]
and as the LHS is decreasing in \(s,\) also
\[\Mh 2^{-j_1\shph} \ge M2^{-j_1\sph}.\]
Using these bounds, we can control the error in \(\fh,\) producing adaptive 
confidence bands for \(f.\)

To construct the bands, we will introduce some more resolution choices 
\(\jh.\) Firstly, we consider the class resolution choice \(\jc,\) chosen in 
terms of the class \(C^s(M).\) For \(\kappa > 0,\) \(\mu \ge 1,\) define
\begin{align}
\notag \jc(\kappa, \mu) &\coloneqq \sup \left(\{\jmin\} \cup \{j > \jmin : 
M2^{-j\sph} \ge \kappa c_{n,\mu}\} \right)\\
\label{eq:jc-defn}
&= \max\left(\jmin, \,\lfloor \log_2( M \kappa c_{n,\mu} ) / (s + \tfrac12) 
\rfloor\right),
\intertext{which we can estimate by}
  \label{eq:jhc-defn}
\jhc(\kappa, \mu) &\coloneqq \max\left(\jmin,\, \lfloor \log_2( \Mh / \kappa 
c_{n,\mu} ) / (\sh + \tfrac12) \rfloor\right).
\end{align}

Secondly, to produce exact confidence bands, we will need the undersmoothed 
resolution choice \(\je.\) Fix \(u_n \in \N,\) \(2^{u_n} \sim \log n,\) and 
set
\[\je(\kappa, \mu) \coloneqq \jc(\kappa, \mu) + \lceil \log_2 \jc(\kappa, \mu) 
\rceil + u_n,\]
defining \(\jhe\) similarly, in terms of \(\jhc.\) Fix \(0 < \delta \le 
\sqrt2\) small, let \(\lu \coloneqq \lambda + \delta,\) and \(\ll \coloneqq 
\lambda - \sqrt2.\)  For convenience, write \(\jc \coloneqq \jc(\lu, 1),\) 
\(\je \coloneqq \je(\ll, 1),\) and likewise \(\jhc,\) \(\jhe.\)

We may now proceed to define our bands. Let
\begin{align*}
  a(j) &\coloneqq \sqrt{2 \log(2) j},\\
  b(j) &\coloneqq a(j) - \frac{\log (\pi \log 2) + \log j - \tfrac12 \log (1 + 
  \upsilon_\varphi)}{2a(j)},\\
  c(j) &\coloneqq \overline \sigma_\varphi n^{-1/2} 2^{j/2},\\
  x(\gamma) &\coloneqq -\log \left(-\log (1-\gamma)\right),\\
  R_1(j, \gamma) &\coloneqq c(j)\left(\frac{x(\gamma)}{a(j)} + b(j)\right),\\
  l(j) &\coloneqq \max(j, \min(\jhc, \jmax)),\\
  R_2(j) &\coloneqq \tau_\varphi \lu (2^{l(j)/2} - 
  2^{j/2})c_{n,\nu}/(1-2^{-1/2}),\\
  R_3(j) &\coloneqq \begin{cases}
    \tau_\varphi, \Mh 2^{-l(j)\sh}/(2^{\sh}-1) & \sh > 0, \\
    \infty, & \sh = 0,
  \end{cases}
\end{align*}
where \(\overline \sigma_\varphi\) is given by \autoref{ass:sigma-maximum},
\begin{equation}
  \label{eq:tau-defn}
\tau_\varphi \coloneqq \sup_{t \in [0, 1]} 2^{-(j_0+1)/2} \sum_{k \in \Z} 
\abs{\psi_{j_0+1,k}(t)} = \sup_{j > j_0} \sup_{t \in [0, 1]} 2^{-j/2} \sum_{k 
\in \Z} \abs{\psi_{j, k}(t)},
\end{equation}
and
\[\upsilon_\varphi \coloneqq -\frac{\sum_{k \in \Z} 
\varphi'(t_0-k)^2}{\overline \sigma_\varphi \sigma_\varphi''(t_0)} .\]

If we set \(\smin > 0,\) \(\nu > 1,\) the undersmoothed resolution choice 
\(\jhe,\) with confidence radius
\begin{equation*}
  \Re \coloneqq R_1(\jhe, \gamma),
\end{equation*}
will be shown to give a band \(\Ce\) satisfying \autoref{thm:exact-band}. If 
instead we set \(\smin = 0,\) \(\nu = 1,\) and define
\begin{equation*}
  \gamma_n \coloneqq \gamma/(\jmax - \jmin + 1),
\end{equation*}
then the adaptive resolution choice \(\jha,\) with confidence radius
\begin{equation*}
  \Ra \coloneqq R_1(\jha, \gamma_n) + R_2(\jha) + R_3(\jha),
\end{equation*}
will be shown to give a band \(\Ca\) satisfying \autoref{thm:adaptive-band}.

\comment{The values of the free parameters can be chosen to minimize the 
bounds on \(\Re\) and \(\Ra,\) but these are only bounds, and do not 
necessarily represent the true performance of the procedure.  Alternatively, 
the parameters could be chosen to maximize empirical performance on a suite of 
example functions.}

\section{Constructive results}
\label{app:constructive}

We now prove our results on the existence of adaptive confidence bands.
To proceed, we will decompose the error in estimates \(\fh(j)\) into variance 
and bias terms,
\[\norm{\fh(j)-f}_\infty \le \norm{\fh(j) - \fb(j)}_\infty + 
\norm{\fb(j)-f}_\infty,\]
where
\[\fb(j) \coloneqq \E_f[\fh(j)] = f_{j_0, j}.\]
To control the variance, we will need the following result from 
\citet{bull_smirnov-bickel-rosenblatt_2011}.

\begin{lemma}
  \label{lem:variance-bound}

  Let \(0 < \gamma_n \le \gamma_0 < 1,\) and \(\gamma_n^{-1} = 
  o(n^{-\alpha}),\) for all \(\alpha > 0.\) Then as \(n \to \infty,\) 
  uniformly in \(f \in L^2([0, 1]),\)
  \[\sup_{j_n \ge \jmin} \abs*{ \gamma_n^{-1}
  \P\left(a(j_n) \left(\frac{\norm{\fh(j_n) - \fb(j_n)}_\infty}{c(j_n)} - 
  b(j_n)\right) > x(\gamma_n)\right) - 1 } \to 0.\]
\end{lemma}

To bound the bias, we must control the estimators \(\jh,\) \(\sh\) and 
\(\Mh.\) We will show that, on events \(E_n\) with probability tending to 1, 
these estimators are close to the quantities they bound.

\begin{lemma}
  \label{lem:j-b-s-bounds}

  Set \(\jl \coloneqq \ja(\lu, \nu),\) \(\ju \coloneqq \ja(\ll, \nu).\)  For 
  \(s \in [\smin, \smax],\) \(M > 0,\) and \(f \in C^s_0(M),\) we have events 
  \(E_n,\) with \(\P(E_n) \to 1\) uniformly, on which:
  \begin{enumerate}
    \item \(\jl \le \jha \le \ju;\)
    \item \(\sh \le s,\) and \(\Mh 2^{-j_1\shph} \ge M2^{-j_1\sph};\) and
    \item \(\sh \ge s_n,\) and \(\Mh \le M_n;\)
  \end{enumerate}
  for sequences \(M_n,\) \(s_n\) satisfying
  \[M_n/M \to \varepsilon^{-1}, \qquad \log_2(n)(s-s_n) \to S,\]
  uniformly over \(f \in C^s_0(M),\) with constant \(S > 0\) depending on 
  \(N,\) \(\varepsilon,\) \(\rho,\) and \(\lambda.\) Also on \(E_n,\) for any 
  \(0 < \kappa \le \lambda + \sqrt2,\) \(1 \le \mu \le \nu\):
  \begin{enumerate}
      \setcounter{enumi}{3}
    \item \(\jhc(\kappa, \mu) \ge \jha\);
    \item \(\jc(\kappa, \mu) \le \jhc(\kappa, \mu) \le \jc(\kappa, \mu) + 
      \Jc(\kappa, \mu);\) and
    \item \(\je(\kappa, \mu) \le \jhe(\kappa, \mu) \le \je(\kappa, \mu) + 
      \Je(\kappa, \mu);\)
  \end{enumerate}
  for sequences \(\Jc(\kappa, \mu), \Je(\kappa, \mu) \to 2S,\) uniformly over 
  \(f \in C^s_0(M).\)
\end{lemma}

\begin{proof}
  For \(n\) such that \(\jmin < \rho ^2j_0,\) set \(E_n \coloneqq \emptyset.\) 
  Otherwise,
  let \(E_n\) be the event that
  \[\sup_{j_0 < j \le \jmax} \sup_{k=0}^{2^j-1} \abs{\bjkh - \bjk} \le 
  \sqrt2\nln,\]
  and if \(n\) is large enough that \(\jl > \jmin,\) also
  \[\abs{\bjkh[4] - \bjk[4]} \le \delta \nln,\]
  for \(j_4,\) \(k_4\) as follows: set \(j_4 \coloneqq \jl,\) and choose 
  \(k_4\) to satisfy \(\abs{\bjk[4]} \ge \lu c_{n, \nu},\)
  which is possible by the definition of \(\jl.\) Now, for \(x > 0,\)
  \(1 - \Phi(x) \le \phi(x)/x,\)
  so we have
  \begin{align*}
    \P(E_n^c) &\le \P\left(\abs{\bjkh[4] - \bjk[4]} > \delta \nln\right) + 
    \sum_{j=j_0}^{\jmax}\sum_{k=0}^{2^j-1} \P\left(\abs{\bjkh - \bjk} > 
    \sqrt2\nln\right)\\
    &\le (\pi \log n)^{-1/2} \left(\sqrt2\delta^{-1}n^{-\delta^2/2} + 
    2^{\jmax+1} n^{-1} \right) \\
    &=O\left((\log n)^{-3/2}\right).
  \end{align*}

  \begin{enumerate}
    \item If \(\jl = \jmin,\) then trivially \(\jha \ge \jl.\) Otherwise, on 
      \(E_n,\)
      \[\abs{\bjkh[4]} \ge \abs{\bjk[4]} - \delta\nln \ge \lambda c_{n,\nu},\]
      and again \(\jha \ge \jl.\) Similarly, for all \(\ju < j \le \jmax, k,\)
      \[\abs{\bjkh} \le \abs{\bjk} + \sqrt2\nln < \lambda c_{n,\nu},\]
      so \(\jha \le \ju.\)

    \item On \(E_n,\) we have
      \[M^s_{i,j} \in [\underline M^s_{i, j}, \overline M^s_{i, j}],\]
      for any \(i \le j \le \jmax.\) If \(s < \smax,\) by the argument given 
      in \autoref{app:construction}, we then obtain
      \[\sh \le s, \qquad \Mh 2^{-j_1\shph} \ge M2^{-j_1\sph}.\]
      If \(s = \smax,\) the results follow similarly, noting that \(\sh \le 
      \smax\) by definition.

    \item On \(E_n,\) \(j_3 = \jha \le \ju \le \jc(\ll, \nu),\) and for \(n\) 
      large \(\jc(\ll, \nu) > \jmin,\) so
      \[d_n \coloneqq c_{n,1}2^{j_3\sph} \le c_{n,\nu}2^{j_3\sph} \le 
      M\ll^{-1},\]
      and also
      \[e_n \coloneqq c_{n,1}2^{j_1\sph} \to 0.\]
      We then obtain
      \[R(s)
      \le \frac{\underline M^s_{j_2, j_3} + 2\sqrt 2d_n} {\overline M^s_{j_0, 
      j_1} - 2\sqrt2e_n}
      \le \frac{M^s_{j_2, j_3} + 2\sqrt 2d_n} {M^s_{j_0, j_1} - 2\sqrt2e_n}
      \le R_n \varepsilon \frac{M^s_{j_2, j_3}} {M^s_{j_0, j_1}} \le R_n,\]
      for a sequence
      \[R_n \to \varepsilon^{-1}(1 + 2\sqrt{2}\ll^{-1}) \eqqcolon R.\]

      On \(E_n,\) \(\sh \le s \le \smax\) by (ii), so if \(\sh = \smax,\) we 
      are done. If not, then \(R(\sh) \ge \varepsilon,\) and
      \[2^{(j_2-j_1)(s - \sh)} \le \frac{\overline M^s_{j_2,j_3}/\overline 
      M^{\sh}_{j_2, j_3}} {\underline M^s_{j_0,j_1}/\underline M^{\sh}_{j_0, 
      j_1}} = \frac{R(s)}{R(\sh)} \le \frac{R_n}{\varepsilon}.\]
      Since
      \[j_2 - j_1 \ge \lfloor \jmin/\rho  \rfloor - j_1 \eqqcolon \delta_n,\]
      we have
      \[\sh \ge s - \log_2(\varepsilon^{-1}R_n) / \delta_n \eqqcolon s_n,\]
      and since \(\delta_n \sim \log_2(n)/\rho (2N+1),\)
      \[\log_2(n) (s - s_n) \to \rho (2N+1) \log_2(\varepsilon^{-1}R) 
      \eqqcolon S.\]
      Likewise,
      \[\Mh \le M(s) \le \varepsilon^{-1}(\underline M^s_{j_0, j_1} + 
      2\sqrt2e_n) \le \varepsilon^{-1}(M + 2\sqrt2e_n) \le M_n,\]
      for a sequence \(M_n > 0,\) with \(M_n/M \to \varepsilon^{-1}.\)
      
    \item If \(\jha = \jmin,\) then trivially \(\jhc(\kappa, \mu) \ge \jha.\) 
      If not, on \(E_n,\) for \(j = \jha,\) we have some \(k\) such that 
      \(\abs{\bjkh} \ge \lambda c_{n,\nu}.\)  Hence
      \[\Mh 2^{-\jha\shph} \ge \varepsilon^{-1}(\lambda + \sqrt2)c_{n,\nu} \ge 
      \kappa c_{n,\mu},\]
      and again \(\jhc(\kappa, \mu) \ge \jha.\)
      
    \item On \(E_n,\) by the above we have
      \[M2^{-(\jhc(\kappa, \mu)+1) \sph} \le \Mh2^{-(\jhc(\kappa, \mu)+1) 
      \shph} < \kappa c_{n,\mu},\]
      and so \(\jhc(\kappa, \mu) \ge \jc(\kappa, \mu).\) Equally, from 
      \eqref{eq:jc-defn}, \eqref{eq:jhc-defn} and the above, we obtain
      \begin{align*}
        \jhc(\kappa, \mu) - \jc(\kappa, \mu) &\le 1 + 2 \log_2 (\Mh/M) + 4 
        \log_2 (\sqrt{n}M/\kappa) (s - \sh)\\
        &\le \Jc(\kappa, \mu),
      \end{align*}
      for a sequence \(\Jc(\kappa, \mu) \to 2S.\)

    \item From (v), we also have
      \[\jhe(\kappa, \mu) - \je(\kappa, \mu) \le \Je(\kappa,\mu),\]
      for a sequence \(\Je(\kappa,\mu) \to 2S.\) \qedhere
  
  \end{enumerate}
\end{proof}

We may now bound the bias of \(\fh\) with the estimators \(\jh,\) \(\sh\) and 
\(\Mh,\) which bound the true parameters by the above lemma.

\begin{lemma}
  \label{lem:bias-bound}
  Let \(j_n \ge \jha.\)
  On events \(E_n\) as in \autoref{lem:j-b-s-bounds}, for any \(s \in [\smin, 
  \smax],\) \(M > 0,\) and \(f \in C^s_0(M),\)
  \[\norm{\fb(j_n)-f}_\infty \le R_2(j_n) + R_3(j_n).\]
\end{lemma}

\begin{proof}
  If \(\sh = 0,\) this is trivial. If not, by \autoref{lem:j-b-s-bounds}, on 
  \(E_n\) we have \(j_n \ge \jha \ge \jl,\) and for \(j \ge j_n,\) 
  \(M2^{-j\sph} \le \Mh 2^{-j \shph}.\)  Thus
  \begin{align*}
    \norm{\fb(j_n)-f}_\infty &= \norm{f_{j_n+1, \infty}}_\infty
    \le \tau_\varphi \sum_{j=j_n+1}^\infty 2^{j/2} \sup_{k=0}^{2^j-1} 
    \abs{\bjk}\\
    &\le \tau_\varphi \left(\sum_{j=j_n + 1}^{l(j_n)} 2^{j/2} \lu c_{n,\nu} + 
    \sum_{j=l(j_n)+1}^\infty \Mh 2^{-j \sh} \right)\\
    &\le R_2(j_n) + R_3(j_n).\qedhere
  \end{align*}
\end{proof}

We are now ready to prove our theorems. First, we consider the exact band 
\(\Ce.\)


\begin{proof}[Proof of \autoref{thm:exact-band}]\ 

  \begin{enumerate}
    \item Define the terms
      \begin{align}
        \notag
        d(j, x) &\coloneqq a(j)\left(\frac{x}{c(j)}-b(j)\right),\\
        \notag
        F(j) &\coloneqq d(j, \norm{\fh(j) - f}_\infty),\\
        \label{eq:g-defn}
        G(j) &\coloneqq d(j, \norm{\fh(j) - \fb(j)}_\infty),\\
        \notag
        H(j) &\coloneqq d\left(j, \norm*{\fh(j)_{\ju+1, \infty} - 
        \fb(j)_{\ju+1, \infty}}_\infty\right).
      \end{align}
      We will show that uniformly in \(j,\) \(F,\) \(G\) and \(H\) are close, 
      and \(H\) is independent of \(\jhe,\) so we may bound \(F(\jhe)\) by 
      \autoref{lem:variance-bound}.

      By definition, \(\sh \ge \smin > 0,\) and \(\jhe \ge \jhc(\ll, 1) \ge 
      \jhc,\) so on the events \(E_n,\) by \autoref{lem:bias-bound},
      \begin{align*}
        \abs{F(\jhe) - G(\jhe)} &\le \frac{a(\jhe)}{c(\jhe)} R_3(\jhe)
        \lesssim \sqrt{\frac{n\jhe}{2^{\jhe}}} \frac{\Mh 2^{-\jhe\sh} 
        }{2^{\sh} - 1}\\
        &\lesssim \sqrt{\frac{\jhe}{\jhc(\ll, 1)}} \left(\jhc(\ll, 1) 
        \log(n)\right)^{-\smin}=o(1),
      \end{align*}
      \comment{
      \begin{align*}
        \abs{F(\jhe) - G(\jhe)} &\le \frac{a(\jhe)}{c(\jhe)} R_3(\jhe)\\
        &\lesssim \sqrt{\frac{n\jhe}{2^{\jhe}}} \frac{\Mh 2^{-\jhe\sh} 
        }{2^{\sh} - 1}\\
        &\lesssim \sqrt{\frac{n\jhe}{\jhc(\ll, 1) \log(n)}} \left(\Mh 
        2^{-\jhc(\ll, 1)\shph} \right)2^{-(\jhe - \jhc(\ll, 1))\sh} \\
        &\lesssim \sqrt{\frac{\jhe}{\jhc(\ll, 1)}} \left(\jhc(\ll, 1) 
        \log(n)\right)^{-\smin}\\
        &=o(1),
      \end{align*}
      }
      since \(\jhc(\ll, 1) \ge \jmin,\) and
      \[\frac{\jhe}{\jhc(\ll, 1)} - 1 = \frac{\log_2 \jhc(\ll, 1) + 
      u_n}{\jhc(\ll, 1)} \le \frac{\log_2 \jmin + u_n}{\jmin} \to 0.\]

      Similarly, for \(j_n \ge \je,\) on \(E_n,\)
      \comment{
      \begin{align*}
        \abs{G(j_n) - H(j_n)} &\le a(j_n) c(j_n)^{-1} \norm*{\fh(j_n)_{j_0, 
        \ju} - \fb(j_n)_{j_0, \ju}}_\infty \\
        &\lesssim a(j_n)c(j_n)^{-1} \sum_{j=j_0}^{\ju} 2^{j/2}
        \sup_k \abs{\bjkh - \bjk}\\
        &\lesssim (nj_n)^{1/2}2^{-j_n/2} 2^{\ju/2} \nln \\
        &\lesssim (\je/\jc(\ll, 1))^{1/2}2^{-(\jc(\ll, 1) - \ju)/2}\\
        &\lesssim 2^{-(\jc(\ll, 1) - \jc(\ll, \nu))/2}\\
        &= o(1),
      \end{align*}
      }
      \begin{align*}
        \abs{G(j_n) - H(j_n)} &\lesssim \frac{a(j_n)}{c(j_n)} 
        \sum_{j=j_0}^{\ju} 2^{j/2}
        \sup_k \abs{\bjkh - \bjk}\\
        &\lesssim (\je/\jc(\ll, 1))^{1/2}2^{-(\jc(\ll, 1) - \ju)/2}\\
        &\lesssim 2^{-(\jc(\ll, 1) - \jc(\ll, \nu))/2} = o(1),
      \end{align*}      since
      \[\jc(\ll, 1) - \jc(\ll, \nu) \ge \frac{\nu - 1}{2\smax + 1} \log_2 
      (\log( n)) \to \infty.\]

      On \(E_n,\) \(\jhe\) depends only on \(\bjkh\) for \(j \le \jha \le 
      \ju,\) and \(H(j)\) depends only on \(\bjkh\) for \(j > \ju,\) so 
      \(H(j)\) is independent of \(\jhe.\)  Hence, given \(x, \varepsilon > 
      0,\) for \(n\) large, and any \(j \ge \je,\)
      \begin{align*}
        \P(F(j) \le x \mid E_n, \jhe = j) &\ge \P(G(j) \le x - \varepsilon 
        \mid E_n, \jhe = j) \\
        &\ge \P(H(j) \le x - 2\varepsilon \mid E_n, \jhe = j)\\
        &=\P(H(j) \le x - 2\varepsilon \mid E_n)\\
        &\ge \P(G(j) \le x - 3\varepsilon \mid E_n)\\
        &\ge \P(G(j) \le x - 3\varepsilon) - \P(E_n^c)\\
        &\ge \exp\left(-e^{-(x - 3\varepsilon)}\right) - o(1).
      \end{align*}
      Likewise,
      \[\P(F(j) \ge x \mid E_n, \jhe = j) \le 
      \exp\left(-e^{-(x+3\varepsilon)}\right) + o(1).\]
      As these results are uniform in \(j \ge \jmin,\) and true for any 
      \(\varepsilon > 0,\) we have
      \[\sup_{j \ge \je}\abs*{\P\left(F(j) \ge x \mid E_n, \jhe = j\right) - 
      \exp\left(-e^{-x}\right)} \to 0.\]

      On \(E_n,\) we have \(\jhe \ge \je,\) so
      \begin{align*}
        \P(F(\jhe) \le x \mid E_n)
        &= \sum_{j=\je}^{\infty} \P(F(j) \le x \mid E_n, \jhe = j)\P(\jhe = j 
        \mid E_n)\\
        &=\left( \exp\left(-e^{-x}\right) + o(1) \right)\sum_{j=\je}^{\infty} 
        \P(\jhe = j \mid E_n)\\
        &=\exp\left(-e^{-x}\right) + o(1).
      \end{align*}
      Since \(\P(E_n) \to 1,\) we obtain
      \(\P(F(\jhe) \le x) \to \exp\left(-e^{-x}\right),\)
      and rearranging,
      \[\P(f \not\in \Ce) \to \gamma.\]
      As the limits are all uniform in \(f,\) the result follows.

    \item
      Let \(\Je \coloneqq \Je(\ll, 1),\) so on \(E_n,\) \(\jhe \le \je + \Je\) 
      by \autoref{lem:j-b-s-bounds}.  For \(n\) large, \(\jc > \jmin,\) so
      \comment{
      \[2^{\jc} \sim \left(\frac{M}{\ll c_{n,1}}\right)^{2/(2s+1)},\qquad 
      2^{\je} \sim \frac{\log(n)^2}{(2s+1)\log(2)} \left(\frac{M}{\ll 
      c_{n,1}}\right)^{2/(2s+1)},\]
      }
      \begin{equation}
      \label{eq:jc-je-limits}
      2^{\jc/2} \approx \left(\frac{M}{c_{n,1}}\right)^{1/(2s+1)},\qquad 
      2^{\je/2} \approx \log(n) 2^{\jc/2},
      \end{equation}
      and
      \comment{
      \begin{align*}
        \Re
        &\le \sigma_\varphi \sqrt{2\log(2)(\je+\Je)} 2^{(\je + \Je)/2} 
        n^{-1/2}(1 + o(1))\\
        &\le \sigma_\varphi
        \sqrt{\frac{2\log(n)^3}{\log(2)}} 2^S \left(\frac{M}{\ll 
        c_{n,1}}\right)^{1/(2s+1)} n^{-1/2} (1+o(1)) \\
        &\le \frac{\sigma_\varphi 2^{S+1/2}}{\sqrt{\log(2)}}
        \left(\frac{M}{\ll}\right)^{1/(2s+1)} r_n(s) (1 + o(1)).
      \end{align*}
      }
      \begin{align*}
        \Re
        &\lesssim \sqrt{\je+\Je} 2^{(\je + \Je)/2} n^{-1/2} \lesssim 
        M^{1/(2s+1)} r_n(s).
      \end{align*}
      As \(\P(E_n) \to 1\) uniformly, and the limits are uniform over \(f \in 
      C^s_0(M),\) the result follows. \qedhere
  \end{enumerate}
\end{proof}

We now move on to the adaptive band \(\Ca.\) As the variance term is no longer 
independent of \(\hat j_n,\) we must use a different method to establish the 
validity of our band. We will instead consider \(\jmax - \jmin + 1\) 
confidence bands, one for each possible choice of \(\hat j_n,\) and show that 
the effect of this change is asymptotically negligible.


\begin{proof}[Proof of \autoref{thm:adaptive-band}]\ 

  \begin{enumerate}
    \item Let \(G(j)\) be given by \eqref{eq:g-defn}.  From 
      \autoref{lem:variance-bound}, we have
      \begin{align*}
        \P(G(\jha) > x(\gamma_n)) &\le \P\left(\exists \ j \in [\jmin, \jmax] 
        : G(j) > x(\gamma_n)\right)\\
        &\le \sum_{j=\jmin}^{\jmax} \P\left(G(j) > x(\gamma_n)\right)\\
        &= (\jmax - \jmin + 1)(1 + o(1))\gamma_n\\
        &= \gamma + o(1).
      \end{align*}
      Rearranging, we get
      \[\P\left(\norm{\fh(\jha) - \fb(\jha)}_\infty > R_1(\jha, 
      \gamma_n)\right) \le \gamma + o(1).\]

      By \autoref{lem:bias-bound}, on the events \(E_n,\)
      \[\norm{\fb(\jha)-f}_\infty \le R_2(\jha) + R_3(\jha)\]
      and by \autoref{lem:j-b-s-bounds}, \(\P(E_n) \to 1.\)  Since
      \[\norm{f - \fh(\jha)}_\infty \le \norm{\fh(\jha) - \fb(\jha)}_\infty + 
      \norm{\fb(\jha) - f}_\infty,\]
      we obtain
      \[\P(f \not\in \Ca) \le \gamma + o(1).\]
      As the limits are uniform in \(f,\) the result follows.

    \item
      Since \(\jha \ge \jmin,\) and \(x(\gamma_n) = O(\log \log n),\)
      \comment{
      \begin{align*}
        x(\gamma_n) &\le - \log(- \log(1 - \gamma/\jmax))\\
        &\le -\log(\gamma/\jmax + o(1/\jmax))\\
        &\le -\log \gamma + \log \jmax + o(1)\\
        &= O(\log \log n),
      \end{align*}
      }
      we have that \(R_1(\jha, \gamma_n)\) is dominated by \(b(\jha)c(\jha).\) 
      Let \(\Jc \coloneqq \Jc(\lu, 1),\) so on \(E_n,\) \(\jha \le \jhc \le 
      \jc + \Jc\) by \autoref{lem:j-b-s-bounds}. For \(n\) large, \(\jc > 
      \jmin,\) so
      \comment{
      \[2^{\jc} \sim \left(\frac{M}{\lu c_{n,1}}\right)^{2/(2s+1)},\]
      and we obtain
      \begin{align*}
        R_1(\jha, \gamma_n)
        &\le \sfb \sqrt{2\log(2)(\jc+\Jc)} 2^{(\jc+\Jc)/2} n^{-1/2}(1 + 
        o(1))\\
        &\le \sfb
        \sqrt{2\log(n)} 2^S \left(\frac{M}{\lu c_{n,1}}\right)^{1/(2s+1)}
        n^{-1/2} (1+o(1)).
      \end{align*}
      Likewise on \(E_n,\) for \(n\) large \(\jc + \Jc \le \jmax,\) so 
      \(l(\jha) = \jhc,\) and
      \begin{align*}
        R_2(\jha)
        &\le \frac{\tau_\varphi\lu}{1-2^{-1/2}} 2^{(\jc+\Jc)/2} c_{n,1}\\
        &\le \frac{\tau_\varphi\lu}{1-2^{-1/2}} 2^S \left(\frac{M}{\lu 
        c_{n,1}}\right)^{1/(2s+1)}c_{n,1}(1+o(1)).
      \end{align*}
      Also for \(n\) large, \(\sh \ge s_n > 0,\) so
      \begin{align*}
        R_3(\jha)
        &\le \frac{\tau_\varphi M_n}{2^{s_n}-1} 2^{-\jc s_n}\\
        &\le \frac{\tau_\varphi \varepsilon^{-1}M}{2^{s}-1} 
        2^S\left(\frac{M}{\lu c_{n,1}}\right)^{-2s/(2s+1)}(1+o(1)),
      \end{align*}
      and thus
      \begin{multline*}
        \Ra \le 2^s\left( \sqrt{2}\sfb + \tau_\varphi \lu  
        \left(\frac{1}{1-2^{-1/2}} + 
        \frac{\varepsilon^{-1}}{2^s-1}\right)\right)\\
        \left(\frac{M}{\lu}\right)^{1/(2s+1)}r_n(s)(1 + o(1)).
      \end{multline*}
}
      by \eqref{eq:jc-je-limits}, we obtain
      \begin{align*}
        R_1(\jha, \gamma_n)
        &\lesssim \sqrt{\jc+\Jc} 2^{(\jc+\Jc)/2} n^{-1/2}
        \lesssim M^{1/(2s+1)} r_n(s).
      \end{align*}
      Likewise on \(E_n,\) for \(n\) large \(\jc + \Jc \le \jmax,\) so 
      \(l(\jha) = \jhc,\) and
      \begin{align*}
        R_2(\jha)
        &\lesssim 2^{(\jc+\Jc)/2} c_{n,1}
        \lesssim M^{1/(2s+1)} r_n(s).
      \end{align*}
      Also for \(n\) large, \(\sh \ge s_n > 0,\) so
      \begin{align*}
        R_3(\jha)
        &\lesssim \frac{M_n}{2^{s_n}-1} 2^{-\jc s_n}
        \lesssim \frac{M^{1/(2s+1)}}{2^{s}-1}r_n(s).
      \end{align*}
      As \(\P(E_n) \to 1\) uniformly, and the limits are uniform over \(f \in 
      C^s_0(M),\) the result follows. \qedhere
  \end{enumerate}
\end{proof}

Finally, we prove our result on confidence bands in density estimation and 
regression.


\begin{proof}[Proof of \autoref{thm:other-models}]
  We can prove the result analogously to \autoref{thm:adaptive-band}. To bound 
  the bias term, we will sketch a version of \autoref{lem:j-b-s-bounds} for 
  the density estimation and regression models. It is possible to also adapt 
  the variance bound \autoref{lem:variance-bound}, as discussed in 
  \citet[\S2]{bull_smirnov-bickel-rosenblatt_2011};
  however, we will provide a weaker bound, as a consequence of our lemma.

  Consider the empirical wavelet coefficents
  \[\hat \alpha_{k} \coloneqq \frac1n \sum_{i=1}^n \varphi_{j_0, k}(X_i), 
  \qquad \hat \beta_{j,k} \coloneqq \frac1n \sum_{i=1}^n \psi_{j, k}(X_i),\]
  in density estimation, or
  \[\hat \alpha_{k} \coloneqq \frac1n \sum_{i=1}^n \varphi_{j_0, k}(x_i)Y_i, 
  \qquad \hat \beta_{j, k} \coloneqq \frac1n \sum_{i=1}^n \psi_{j, 
  k}(x_i)Y_i,\]
  in regression. To prove the lemma, we must find an event \(E_n\) on which, 
  with high probability, these estimates are close to the true wavelet 
  coefficients \(\alpha_k,\) \(\beta_{j, k}.\) In density estimation, we use 
  Bernstein's inequality, noting that, for \(j > j_0,\) \(k \in [N, 2^j - 
  N),\) the empirical wavelet coefficients satisfy
  \[\E[\hat \beta_{j, k}] = \beta_{j, k}, \qquad \Var[\hat \beta_{j, k}] \le 
  \frac{\norm{f}_\infty}{n}, \qquad \abs{\hat \beta_{j, k}} \le 
  2^{j/2}\norm{\psi}_\infty,\]
  with similar bounds for the other coefficients.
  
  The regression model is often identified with the white noise model, for 
  \(f\) in classes \(C^s(M),\) \(s \ge \tfrac12\) 
  \citep{brown_asymptotic_1996}. In this case, however, we wish to consider 
  functions with unbounded H\"{o}lder norm, so we must discuss regression 
  explicitly. To control the empirical wavelet coefficients, we use a Gaussian 
  tail bound, noting that for \(j,\) \(k\) as before,
  \[\hat \beta_{j, k} \sim N\left(\frac1n \sum_{i=1}^n \psi_{j,k}(x_i)f(x_i),
  \frac{\sigma^2}{n^2} \sum_{i=1}^n \psi_{j, k}(x_i)^2\right).\]
  For \(j \le \jmax,\) as \(n \to \infty,\) the mean and variance are thus
  \[\beta_{j, k} + O(n^{-1/2}\norm{f}_{C^{1/2}}) \qquad \text{and} \qquad 
  \sigma^2n^{-1}(1 + o(1)),\] uniformly. Again, similar results hold for the 
  other coefficients.

  We thus, in both cases, have events \(E_n\) comparable to those in 
  \autoref{lem:j-b-s-bounds}, but with bounds on wavelet coefficients now 
  depending on the unknowns \(\norm{f}_\infty\) and \(\norm{f}_{C^{1/2}}.\) We 
  will bound them with statistics
  \[T \coloneqq C\norm{\fh(j_1)}_{C^{\smax}} + D,\]
  for constants \(C,\) \(D > 0.\) In density estimation, for \(C,\) \(D\) 
  large this satisfies
  \[\sup_{f \in \mathcal F} \P_f(T < \norm{f}_\infty) \to 0,\]
  and likewise in regression,
  \[\sup_{f \in \mathcal F} \P_f(T < \norm{f}_{C^{1/2}}) \to 0.\]
  In either model, for \(s \in [\smin, \smax],\) \(M > 0,\)
  \[\sup_{f \in C^s_0(M)} \P_f(T > CM + D + 1) \to 0.\]
  We may thus replace \(\norm{f}_\infty,\) or \(\norm{f}_{C^{1/2}},\) with 
  \(T\) in the above, obtaining an analogue of \autoref{lem:j-b-s-bounds} 
  which holds for all \(f \in \mathcal F.\)

  We therefore obtain a bound on the bias term, as in 
  \autoref{thm:adaptive-band}. To bound the variance term, we note that on the 
  event \(E_n,\) we have
  \[\norm{\fh(j_n) - \fb(j_n)}_\infty = O(2^{j_n/2}c_{n,1}),\]
  uniformly in all \(j_n \le \jmax;\) we may then proceed as before.
\end{proof}

\section{Negative results}
\label{app:negative}

We now prove our negative results.  First, we will need a testing inequality 
for normal means experiments, arguing as in \citet{ingster_minimax_1987}.  We 
will prove a modified result, which controls the performance of tests also 
under small perturbations of the means.

\begin{lemma}
  \label{lem:normal-mean-test}
  Suppose we have independent observations \(X_1, \dots, X_n,\) and \(Y_1, 
  Y_2, \dots,\) and we wish to test the hypothesis
  \[H_0: X_i, Y_i \sim N(0, 1),\]
  against alternatives
  \[H_k(\nu): X_i \sim N(\mu \delta_{ik}, 1), \ Y_i \sim N(\nu_i, 1),\]
  for \(k = 1, \dots, n,\) and \(\mu, \nu_i \in \R\), \(\norm{\nu}^2 \le 
  \xi^2.\)  Let \(T = 0\) if we accept \(H_0,\) or \(T = 1\) if we reject.  
  There is a choice of \(k,\) not depending on \(\nu,\) for which the sum of 
  the Type I and Type II errors satisfies
  \[\P_{H_0}(T = 1) + \inf_{\norm{\nu}^2 \le \xi^2} \P_{H_k(\nu)}(T = 0) \ge 1 
  - n^{-1/2}(e^{\mu^2}-1)^{1/2} - (e^{\xi^2}-1)^{1/2}.\]
\end{lemma}

\begin{proof}
  Consider first the case \(\nu = 0.\) The density of \(\P_{H_k(0)}\) w.r.t.\ 
  \(\P_{H_0}\) is
  \[Z_k \coloneqq e^{\mu X_k - \mu^2/2}.\]
  Let \(Z \coloneqq n^{-1} \sum_{k=1}^n Z_k.\) Then\comment{
    \[\E_{H_0} Z = n^{-1} \sum_{k=1}^n \E_{H_0} e^{\mu X_k - \mu^2/2} = 1,\]
  and
  \[\E_{H_0} Z^2 = n^{-2} \sum_{k, l=1}^n \E_{H_0} e^{\mu (X_k + X_l) - \mu^2} 
  = 1 + n^{-1}(e^{\mu^2}-1),\]}
  \(\E_{H_0} Z = 1,\) and \(\E_{H_0} Z^2 = 1 + n^{-1}(e^{\mu^2}-1),\)
  so
  \[\E_{H_0} (Z - 1)^2 = \Var_{H_0}Z = n^{-1}(e^{\mu^2}-1).\]
  We thus have
  \begin{align*}
    \P_{H_0}(T = 1) + \max_{k=1}^n \P_{H_k(0)}(T = 0) &\ge
    \P_{H_0}(T = 1) + n^{-1}\sum_{k=1}^n \P_{H_k(0)}(T = 0)\\
    &= 1 + \E_{H_0}[(Z-1)1(T=0)]\\
    &\ge 1 - \Var_{H_0}(Z)^{1/2}\\
    &= 1 - n^{-1/2}(e^{\mu^2}-1)^{1/2}.
  \end{align*}

  Fix \(k\) maximizing the above expression, and consider a hypothesis 
  \(H_k(\nu)\) with \(\norm{\nu}^2 \le \xi^2.\)  The density of 
  \(\P_{H_k(\nu)}\) w.r.t.\ \(\P_{H_k(0)}\) is
  \[Z' \coloneqq e^{\sum_i \nu_i Y_i - \norm{\nu}^2/2},\]
  and similarly we have
  \[\E_{H_k(0)} (Z'-1)^2 = \Var_{H_k(0)} Z' = e^{\norm{\nu}^2} - 1.\]
  Thus
  \begin{multline*}
    \P_{H_0}(T=1) + \P_{H_k(\nu)}(T=0) \\
    \begin{aligned}
      &= \P_{H_0}(T=1) + \P_{H_k(0)}(T=0) + \E_{H_k(0)}[(Z'-1)1(T=0)]\\
      &\ge \P_{H_0}(T=1) + \P_{H_k(0)}(T=0) - \Var_{H_k(0)}[Z']^{1/2}\\
      &\ge 1 - n^{-1/2}(e^{\mu^2}-1)^{1/2} - (e^{\xi^2}-1)^{1/2}.
    \end{aligned}
  \end{multline*}
  As this is true for all \(\norm{\nu}^2 \le \xi^2,\) the result follows.
\end{proof}

We may now prove our result on minimax rates in \(C^s_0(M).\) For \(f \in 
C^s(M),\) the argument is standard (see, for example, 
\citealp[\S2.6.2]{tsybakov_introduction_2009}), but we must check that we can 
construct suitable alternative hypotheses lying within the restricted class 
\(C^s_0(M).\)

\begin{proof}[Proof of \autoref{thm:minimax-rate}]
  Suppose such an estimator \(\fh_n\) exists. For \(i > 0,\) set \(j_{i+1} 
  \coloneqq \rho j_i + 1,\) and consider functions
  \[f_0 \coloneqq \beta_{j_0} \varphi_{j_0, 0} + \sum_{i=1}^\infty \beta_{j_i} 
  \psi_{j_i, 0}, \qquad f_k \coloneqq f_0 + \beta_{j} \psi_{j, k},\]
  where \(\beta_j \coloneqq M2^{-j\sph},\) \(j > j_0\) is to be determined, 
  and \(k \in [N, 2^{j}-N).\)  By definition, these functions are in 
  \(C^s_0(M)\). By standard arguments, \(\fh_n\) must be able to distinguish 
  the hypothesis \(H_0 : f = f_0\) from alternatives \(H_k : f = f_k,\) 
  contradicting \autoref{lem:normal-mean-test}.
\end{proof}
  
\comment{  We will use the estimator \(\fh_n\) to test the hypothesis \(H_0 : 
f = f_0\) against alternatives \(H_k : f = f_k,\) and show this contradicts 
\autoref{lem:normal-mean-test}.

\marginpar{Can just cite the below if needed.}
  Define the test
  \[T_n \coloneqq 1(\norm{\fh_n - f_0}_\infty \ge \tfrac12M\norm{\psi}_\infty 
  2^{-js}),\]
  and set \(\delta \coloneqq \tfrac14(1 - 2\gamma).\) For \(n\) and \(C > 0\) 
  large, we have
  \[\sup_{f \in C^s_0(M)} \P_f(\norm{\fh_n - f}_\infty \ge Cr_n) \le \gamma + 
  \delta.\]
  Allowing \(j \to \infty,\) choose \(n\) so that
  \[Cr_n \sim \tfrac14 M\norm{\psi}_\infty 2^{-js} = \tfrac14 \norm{f_k - 
  f_0}_\infty.\]
  We then obtain, for \(j\) large,
  \begin{equation}
    \label{eq:test-good}
    \P_{f_0}(T_n=1) + \sup_{k}\P_{f_k}(T_n=0) \le 2(\gamma + \delta) = 1 - 
    2\delta.
  \end{equation}

  We now apply \autoref{lem:normal-mean-test}, letting the observations 
  \(X_i\) correspond to \(\langle \psi_{j,k}, Y \rangle,\) for all possible 
  choices of \(k,\) and the \(Y_i\) to the other empirical wavelet 
  coefficients. We have
  \[n = o\left(j2^{j(2s + 1)}\right),\]
  so the quantity
  \[\mu^2 = nM^2 2^{-j(2s+1)} = o(j),\]
  and \(\xi^2 = 0.\) We thus obtain, for \(j\) large,
  \[1 - (2^{j}-2N)^{-1/2}(e^{\mu^2}-1)^{1/2} + (e^{\xi^2}-1)^{1/2} \ge 1 - 
  \delta,\]
  contradicting \eqref{eq:test-good}.
\end{proof}
  }

Finally, we will show that the self-similarity condition 
\eqref{eq:self-similar} is as weak as possible.


\begin{proof}[Proof of \autoref{thm:minimal-conditions}]
  We argue in a similar fashion to \autoref{thm:minimax-rate}, taking care to 
  account for the dishonesty of \(C_n\). Suppose such a band \(C_n\) exists.  
  For \(m = 1, 2, \dots, \infty,\) we will construct functions \(f_m\) which 
  serve as hypotheses for the function \(f.\) We will choose these functions 
  so that \(f_m \in C^{s_m}_1(M),\) for a sequence \(s_m \in (\smin, \smax)\) 
  with limit \(s_\infty \in (\smin, \smax).\) We will then find a subsequence 
  \(n_m\) such that, for \(\delta \coloneqq \tfrac14(1 - 2\gamma),\)
  \[\inf_{m = 2}^\infty \P_{f_\infty} (f_\infty \not\in C_{n_m}) \ge \gamma + 
  \delta,\]
  contradicting our assumptions on \(C_n.\)
   
  Taking infimums if necessary, we may assume \(\rho_j\) increasing; for \(i > 
  0,\) set \(j_{i+1} \coloneqq
  \rho _{j_i}j_i + 1.\) Then for \(m = 1, 2, \dots, \infty,\) set
  \[f_m \coloneqq b_{0, m} \varphi_0 + \sum_{i=1}^\infty b_{i, m} \psi_i + 
  \sum_{l=1}^m b'_l \psi'_l,\]
  where
  \[\varphi_0 \coloneqq \varphi_{j_0, 2^{j_0-1}}, \qquad \psi_i \coloneqq 
  \psi_{j_i, 2^{j_i-1}}, \qquad \psi'_l \coloneqq \psi_{j_{i_l}, k_l},\]
  and \(b_{i, m}, b'_l \in \R,\) \(i_l \in \N,\) and \(k_l \in [N, 2^{j_i}-N) 
  \setminus \{2^{j_i-1}\}\) are to be determined. We will set \(-1 = i_0 < i_1 
  < \dots,\)
  \[b_{i, m} \coloneqq \begin{cases}
    M2^{-j_i(s_l+1/2)}, & i_l < i \le i_{l+1}\text{ for some } l < m, \\
    M2^{-j_i(s_m+1/2)}, & i > i_m,
  \end{cases}\]
  and
  \[b'_l \coloneqq M2^{-j_{i_l}(s_l+1/2)}.\]

  Set  \begin{align*}
    s_0 &\coloneqq \smax, & s_m &\coloneqq s_{m-1} - (j_{i_m}^{-1} - 
    j_{i_m+1}^{-1})\log_2(\varepsilon^{-1}), \quad m > 0,\\
    t_0 &\coloneqq \smin, & t_m &\coloneqq s_m - j_{i_m+1}^{-1} 
    \log_2(\varepsilon^{-1}), \quad m > 0,
  \end{align*}
  and choose \(i_1\) large enough that:
  \begin{enumerate}
    \item \(t_1 > t_0;\)
    \item for \(i \ge i_1,\) the \(\psi_i\) are interior wavelets, supported 
      inside \( (0, 1) \); and
    \item the set of choices for \(k_1\) is non-empty.
  \end{enumerate}
  By definition, \(s_m\) is decreasing, \(t_m\) increasing, and \(s_m - t_m 
  \searrow 0.\) For \(m \ge 1,\) both sequences thus lie in \((\smin, 
  \smax),\) and tend to a limit \(s_\infty \in (\smin, \smax).\)  For all \(m 
  = 1, 2, \dots, \infty,\) \(l \in \N,\) and \(i_l \le i \le i_{l+1},\)
  \begin{equation*}
    M2^{-j_i(s_l+1/2)} \ge \varepsilon M2^{-j_i(t_{l+1} + 1/2)} \ge 
    \varepsilon M 2^{-j_i(s_m+1/2)},
  \end{equation*}
  so indeed \(f_m \in C^{s_m}_1(M).\) 

  We have thus defined \(f_1,\) making an arbitrary choice of \(k_1\); for 
  convenience, set \(n_1 = 1.\) Inductively, suppose we have defined 
  \(f_{m-1}\) and \(n_{m-1},\) and set \(r_n \coloneqq r_n(s_{m-1}).\) For 
  \(n_m > n_{m-1}\) and \(D > 0\) both large, we have:
  \begin{enumerate}
    \item \(\P_{f_{m-1}}(f_{m-1} \not\in C_{n_m}) \le \gamma + \delta\); and
    \item \(\P_{f_{m-1}}(\abs{C_{n_m}} \ge Dr_{n_m}) \le \delta.\)
  \end{enumerate}
  Setting \(T_n = 1\left(\exists\ f \in C_n : \norm{f - f_{m-1}}_\infty \ge 
  2Dr_n\right),\) we then have
  \begin{align}
    \notag \P_{f_{m-1}}(T_{n_m}=1) &\le \P_{f_{m-1}}(f_{m-1} \not \in C_{n_m}) 
    + \P_{f_{m-1}}(\abs{C_{n_m}} \ge Dr_{n_m})\\
    \label{eq:c-accurate}
    &\le \gamma + 2\delta.
  \end{align}
  We claim it is possible to choose \(f_m\) and \(n_m\) so that also, for any 
  further choice of functions \(f_l,\)
  \begin{equation}
    \label{eq:f-separate}
    \norm{f_\infty - f_{m-1}}_\infty \ge 2Dr_{n_m},
  \end{equation}
  and
  \begin{equation}
    \label{eq:f-indistinguishable}
    \P_{f_\infty}(T_{n_m} = 0) \ge 1 - \gamma - 3\delta = \gamma + \delta.
  \end{equation}
  We may then conclude that
  \[\P_{f_\infty}(f_\infty \not\in C_{n_m}) \ge \P_{f_\infty}(T_{n_m} = 0) \ge 
  \gamma + \delta,\]
  as required.

  It remains to verify the claim. Letting \(i_m \to \infty,\) choose \(n_m\) 
  so that
  \begin{equation}
    \label{eq:n-defn}
    r_{n_m} \sim D'2^{-j_{i_m}s_m},
  \end{equation}
  for \(D' > 0\) to be determined. Now,
  \begin{align*}
    D''(i_m) &\coloneqq \sum_{l=m}^\infty \left(2^{-j_{i_{l+1}}s_{l+1}} + 
    \sum_{i=i_l+1}^{i_{l+1}} 2^{-j_is_l} \right)\\
    &\le \sum_{l=m}^\infty \left(2^{-j_{i_{l+1}}\smin} + 
    \sum_{i=i_l+1}^{i_{l+1}} 2^{-j_i\smin} \right)\\
    &\le 2\sum_{j=j_{i_m+1}}^\infty 2^{-j\smin}\\
    &= \frac{2^{1-j_{i_m+1}\smin}}{1-2^{-\smin}},
  \end{align*}
  so, for \(i_m\) large,
  \begin{align*}
    \norm{f_{m-1} - f_\infty}_\infty &\ge \norm{b'_m \psi'_m}_\infty - 
    \norm*{\sum_{l=m+1}^\infty b'_l \psi'_l + \sum_{i=i_m+1}^\infty \left( 
    b_{i,\infty} - b_{i, m-1}\right)\psi_i}_\infty\\
    &\ge M\norm{\psi}_\infty \left( 2^{-j_{i_m}s_m} -
    D''(i_m)\right)\\
    &\ge  M \norm{\psi}_\infty \left( 2^{-j_{i_m}s_m}- 
    \frac{2^{1-j_{i_m+1}\smin}}{1-2^{-\smin}}\right)\\
    &\ge \tfrac12 M\norm{\psi}_\infty 2^{-j_{i_m}s_m}.
  \end{align*}
  We have thus satisfied \eqref{eq:f-separate}, for a suitable choice of 
  \(D'.\) 

  To satisfy \eqref{eq:f-indistinguishable}, we will apply 
  \autoref{lem:normal-mean-test}, testing \(H_0 : f = f_{m-1}\) against 
  \(H_1:f=f_\infty.\) The observations \(X_i\) will correspond to \(\int 
  \psi_m'(t)\, dY_t,\) for all possible choices of \(k_m,\) and the \(Y_i\) to 
  the other empirical wavelet coefficients.  From \eqref{eq:n-defn},
  \[n_m = O\left(j_{i_m}2^{j_{i_m}(2 + s_{m-1}^{-1})s_m}\right),\]
  so the quantity
  \begin{align*}
    \mu^2 &= n_m(b'_m)^2 = n_mM^2 2^{-j_{i_m}(2s_m+1)}\\
    &= O\left(j_{i_m}2^{j_{i_m} \left(s_m/s_{m-1} - 1\right)}\right)\\
    &= O\left(j_{i_m} \varepsilon^{(j_{i_m}/j_{i_m-1} - 1)/s_{m-1}}\right)\\
    &= o(j_{i_m}),
  \end{align*}
  and likewise
  \begin{align*}
    \xi^2 &= n_m \sup_{f_\infty} \left( \sum_{l=m}^\infty (b'_{l+1})^2 + 
    \sum_{i={i_m}+1}^\infty (b_{i,m-1}-b_{i,\infty})^2\right)\\
    &\le n_mM^2\sum_{l=m}^\infty\left(2^{-j_{i_{l+1}}(2s_{l+1}+1)} + 
    \sum_{i={i_l}+1}^{i_{l+1}} 2^{-j_i(2s_l+1)}\right)\\
    &= O\left(n_m2^{-j_{i_m+1}(2s_m+1)}\right)\\
    &= O\left(j_{i_m}2^{j_{i_m}s_m/s_{m-1}-j_{i_m+1}}\right)\\
    &= o(1).
  \end{align*}
  Thus, for \(i_m\) large,
  \[(2^{j_{i_m}}-(2N+1))^{-1/2}(e^{\mu^2}-1)^{1/2} + (e^{\xi^2}-1)^{1/2} \le 
  \delta.\]
  Hence by \autoref{lem:normal-mean-test}, if we take \(i_m\) large enough 
  also that \eqref{eq:c-accurate} holds, then \eqref{eq:f-indistinguishable} 
  holds for a suitable choice of \(k_m,\) and our claim is proved.
\end{proof}

\bibliographystyle{abbrvnat}
{\footnotesize \bibliography{hacb}}

\end{document}